\documentclass[11pt]{amsart}

\usepackage[pdftex]{graphicx}
\usepackage{color}

\usepackage[margin=1in]{geometry}

\usepackage{amsfonts}
\usepackage{amsmath}
\usepackage{amssymb}
\usepackage{amsthm}
\usepackage{subfig}
\usepackage{float}
\usepackage{MnSymbol}
\usepackage{extarrows}

\usepackage{paralist}
\usepackage[colorlinks,cite color=blue,pagebackref=true,pdftex]{hyperref}

\usepackage[T1]{fontenc}

\renewcommand\emptyset{\varnothing}

\newtheorem{thm}{Theorem}
\newtheorem{cor}[thm]{Corollary}
\newtheorem{lem}[thm]{Lemma}
\newtheorem{prop}[thm]{Proposition}

\theoremstyle{definition}

\newtheorem{example}[thm]{Example}
\newtheorem{rem}[thm]{Remark}

\title{Divisors on Projective Hibi Varieties}

\author{Tobias Friedl}
\address{Fachbereich Mathematik und Informatik, %
Freie Universit\"at Berlin, %
Germany}
\email{tfriedl@zedat.fu-berlin.de}


\date{\today}
\thanks{The author has been supported by European Research Council
under the European Union's Seventh Framework Programme (FP7/2007-2013) / ERC
grant agreement n$^\mathrm{o}$ 247029.}
\begin{document}

\begin{abstract}
We compute the divisor class group and the Picard group of projective varieties with Hibi rings as homogeneous coordinate rings. These varieties are precisely the toric varieties associated to order polytopes. We use tools from the theory of toric varieties to get a description of the two groups which only depends on combinatorial properties of the underlying poset.
\end{abstract}

\maketitle

\section{introduction}

\newcommand{\Po}{\mathcal{P}}
\newcommand{\IL}{\mathcal{I}(\mathcal{P})}

Let $(\mathcal{P},\le)$ be a finite partially ordered set (poset). A subset $I\subseteq\Po$ is called an \emph{order ideal} if it is down-closed, i.e. $p\in I$ and $q\le p$ implies $q\in I$. Denote by $\IL$ the set of all order ideals of $\Po$. The poset $(\IL,\subseteq)$ is a distributive lattice with join $I\vee J = I\cup J$ and meet $I\wedge J = I\cap J$ for $I,J\in\IL$. 

\emph{Hibi rings} \cite{Hi87} are graded algebras with straightening laws associated to finite posets. More precisely, for a poset $\mathcal{P}=\{p_1,\ldots,p_n\}$ the Hibi ring $\mathbb{C}[\mathcal{P}]$ is the subalgebra of $\mathbb{C}[x_1,...,x_n,t]$ generated by the set of monomials $\{t\prod_{p_i\in I}{x_i}:I\in\mathcal{I}(\mathcal{P})\}$. Hibi rings are normal Cohen-Macaulay domains and we have $\mathbb{C}[\mathcal{P}]\cong\mathbb{C}[y_I:I\in\mathcal{I}(\mathcal{P})]/\mathfrak{I}_{\mathcal{I}(\mathcal{P})}$, where $\mathfrak{I}_{\mathcal{I}(\mathcal{P})}$ is the ideal generated by the so-called \emph{Hibi relations} $y_Iy_J-y_{I\wedge J}y_{I\vee J}$ for all $I,J\in\mathcal{I}(\mathcal{P})$ (see \cite{Hi87}). 


Since the Hibi relations are homogeneous there is a natural grading on $\mathbb{C}[\mathcal{P}]$ coming from the standard grading on $\mathbb{C}[y_I:I\in\mathcal{I}(\mathcal{P})]$. In the following our central object of sudy are the projective varieties $X_\mathcal{P}$ with the graded ring $\mathbb{C}[\mathcal{P}]$ as its homogeneous coordinate ring, which we will call $\emph{(projective) Hibi varieties}$. Hibi varieties appear for example as flat degenerations of Grassmannians and flag varieties (\cite{MS05},\cite{EH12}). Moreover, they generalize several well-studied classes of varieties, such as certain determinantal and ladder determinantal varieties (\cite{BC03},\cite{Co95}).

Hibi varieties are toric varieties, hence geometric questions can be reduced to discrete-geometric questions about polytopes and fans. In the case of Hibi varieties one can hope to go even one step further and describe the geometry of $X_\mathcal{P}$ in terms of the combinatorics of $\mathcal{P}$.
A first step was done by Wagner in \cite{Wa96}, where the orbits of the torus action and the singular locus of $X_\mathcal{P}$ are described in terms of properties of $\mathcal{P}$.

 In the present paper we compute the divisor class group and the Picard group of Hibi varieties. In Section 2 we describe the polytope of $X_\mathcal{P}$. This was already used without proof in \cite{Wa96}. In Section 3 we use general results on toric varieties to compute the divisor class group of $X_\mathcal{P}$. Finally, in Section 4 we use the description of the divisor class group to compute the Picard group of $X_\mathcal{P}$.

\section{Hibi Varieties and Order Polytopes}

\newcommand{\Pop}{\mathcal{P}^{op}}

Let $\mathcal{P}$ be a finite poset. The projective variety $X_\mathcal{P}=\textnormal{Proj}(\mathbb{C}[\mathcal{P}])$ is called the \emph{(projective) Hibi variety} associated to $\mathcal{P}$.
Hibi varieties appear in various contexts and generalize some well-studied classes of varieties, as the following examples show.

\begin{example}
Let $\mathcal{P}_n$ denote the chain consisting of $n$ elements. The Hibi variety $X_{\mathcal{P}_n}$ is the complex projective space $\mathbb{P}^n$. More generally, if $\mathcal{P}$ is the disjoint union of chains $\mathcal{P}_{n_1},\ldots,\mathcal{P}_{n_l}$ the associated Hibi variety $X_\mathcal{P}$ is the Segre embedding of $\mathbb{P}^{n_1}\times\cdots\times\mathbb{P}^{n_l}$.
\end{example}

\begin{example}
For $1\le d\le n$ there exists a flat degeneration taking the \emph{Grassmannian} $G_{d,n}$ of $d$-dimensional subspaces of an $n$-dimensional complex vector space to the Hibi variety $X_{\mathcal{P}_d\times\mathcal{P}_{n-d}}$. For details see \cite{EH12}, \cite{Fr13} or \cite{St96}. More generally, also \emph{flag varieties} degenerate to Hibi varieties (see \cite{MS05}).
\end{example}

\begin{example}
\emph{Projective determinantal varieties} are determined by the vanishing of all minors of a fixed size of a matrix of indeterminates. In the case of $2$-minors of an $(n\times m)$-matrix $A$ the determinantal variety is the Hibi variety associated to $\mathcal{P}_{n-1}\cupdot\mathcal{P}_{m-1}$. Indeed, the lattice $\mathcal{I}(\mathcal{P}_{n-1}\cupdot\mathcal{P}_{m-1})$ is isomorphic to $\mathcal{P}_n\times\mathcal{P}_m$ and Hibi relations in $\mathcal{P}_n\times\mathcal{P}_m$ correspond precisely to the $2$-minors of $A$.
\end{example}

\begin{example}
\emph{Ladder determiantal varieties} are a generalisation of determinantal varieties, where instead of matrices so-called \emph{ladders} of indeterminates are considered (see e.g. \cite{Co95}). In the case of $2$-minors, these are again Hibi varieties.
\end{example}

In the following we will describe the polytope associated to the toric variety $X_\mathcal{P}$. For a poset $\Po$ a subset $J\subseteq \Po$ is called an \emph{order filter} if it is up-closed, i.e. if $b\ge a$ and $a\in J$ implies $b\in J$. Note that $J$ is an order filter if and only if its complement $\Po\backslash J$ is an order ideal. The set $\mathcal{J}(\Po)$ of all order filters is a distributive lattice with union and intersection as join and meet operation, respectively. We have $\mathcal{J}(\Po)\cong\mathcal{I}(\Pop)$, where $\Pop$ is the \emph{opposite poset} of $\Po$, the poset with the same underlying set as $\Po$ but with the order reversed. 

For a subset $S\subset\Po$ we denote by $\mathbf{a}_S\in\mathbb{R}^\mathcal{P}$ the characteristic vector of $S$, i.e. $a_p=1$ if $p\in S$ and $a_p=0$ otherwise. The convex hull of the set $\{\mathbf{a}_J:J\in\mathcal{J}(\mathcal{P})\}$ is called the \emph{order polytope} of $\mathcal{P}$ and denoted by $\mathcal{O}(\mathcal{P})$. It can be shown that $\mathcal{O}(\mathcal{P})$ consists of all order-preserving functions $f:\mathcal{P}\to[0,1]\subseteq\mathbb{R}$ (see \cite{St86}). 
There is the following close connection between Hibi varieties and order polytopes.



\begin{prop}
The Hibi variety $X_\mathcal{P}$ is isomorphic to the projective toric variety associated to the order polytope $\mathcal{O}(\Pop)$.
\end{prop}
\begin{proof}
We will only sketch the proof, using results and notation from \cite{CLS11}. As in Chapter 1 and 2 of \cite{CLS11} for a finite set of lattice points $\mathcal{A}=\{\mathbf{a}_1,\ldots,\mathbf{a}_m\}\subseteq\mathbb{Z}^k$ we denote by $Y_\mathcal{A}$ the associated affine toric variety defined to be the Zariski closure of the image of the map
\begin{equation*}
\Phi_\mathcal{A}:(\mathbb{C}^*)^k\to \mathbb{C}^m, 
\mathbf{t}\mapsto (\mathbf{t}^{\mathbf{a}_1},...,\mathbf{t}^{\mathbf{a}_m}).
\end{equation*}
Moreover, let $X_\mathcal{A}$ be the Zariski closure of the image of $\pi\circ\Phi_\mathcal{A}$, where $\pi:(\mathbb{C}^*)^m\to\mathbb{P}^{m-1}$ denotes the canonical projection. 

The order polytope $\mathcal{O}(\Pop)$ is normal, since it has a unimodular triangulation (\cite{St86}). Using that the integral points of $\mathcal{O}(\Pop)$ are precisely its vertices, it now follows that the associated toric variety is isomorphic to $X_\mathcal{A}$, where $\mathcal{A}=\{\mathbf{a}_I:I\in\mathcal{I}(\mathcal{P})\}\subseteq\mathbb{Z}^\mathcal{P}$ and $\mathbf{a}_I$ denotes the characteristic vector of the ideal $I$. For the set $\mathcal{A}'=\{(1,\mathbf{a}_I):I\in\mathcal{I}(\mathcal{P})\}\subseteq\mathbb{Z}^{|\mathcal{P}|+1}$ of lattice points of the homogenization of $\mathcal{O}(\Pop)$ we clearly have $X_\mathcal{A}=X_{\mathcal{A}'}$.
On the other hand, $\mathcal{A}'$ forms a set of generators of the affine semigroup of the Hibi ring $\mathbb{C}[\mathcal{P}]$. Hence if follows from Proposition 2.1.4 in \cite{CLS11} and the quotient description of $\mathbb{C}[\mathcal{P}]$ that $X_{\mathcal{A}'}\cong X_\mathcal{P}$.
\end{proof}

Since $\mathcal{O}(\Pop)$ is full-dimensional and Hibi rings are normal (see \cite{Hi87}), we have the following immediate corollary.

\begin{cor}
$X_\mathcal P$ is a projectively normal toric variety of dimension $|\mathcal{P}|$.
\end{cor}

\section{Divisor Class Group}

A relation $p<q$ with $p,q\in\Po$ is called a \emph{covering relation} if there is no $r\in\Po$ with $p<r<q$. We write $\mathcal{C}(\mathcal P)$ for the set of covering relations in $\mathcal P$. The \emph{Hasse diagram} of $\Po$ is the directed graph on the elements of $\Po$ with an edge from $p$ to $q$ if and only if $p<q\in\mathcal{C}(\mathcal P)$.

For a finite poset $\mathcal{P}$ denote by $\hat{\mathcal{P}}$ the poset obtained from $\mathcal{P}$ by attaching a minimal element $\hat 0$ and a maximal element $\hat{1}$.  For a covering relation $p<q\in\mathcal{C}(\hat{\mathcal P})$ define $\mathbf{u}_{p<q}\in\mathbb{Z}^{\mathcal{P}}$ by
\begin{equation}\label{FacetNormals}
\mathbf{u}_{p<q}=
\begin{cases}
\mathbf{e}_p & \textnormal{ if }q=\hat 1\\
-\mathbf{e}_q & \textnormal{ if }p=\hat 0\\
\mathbf{e}_p-\mathbf{e}_q & \textnormal{ otherwise},
\end{cases}
\end{equation}
where $\mathbf{e}_p$ is the standard basis vector corresponding to an element $p\in\mathcal{P}$. 
Note that these vectors are precisely the facet normals of the order polytope $\mathcal{O}(\Pop)$ (see \cite{St86}). For each such facet normal we can associate a torus-invariant divisor $D_{p<q}$ on $X_\mathcal{P}$. Moreover, the set $\{D_{p<q}:p<q\in\mathcal{C}(\hat{\mathcal{P}})\}$ of all such divisors forms a basis of $\textnormal{Div}_T(X_\mathcal{P})$, the group of torus-invariant divisors on $X_\mathcal{P}$ (see \cite{CLS11}, Chapter 4).

\begin{rem}
The facet of $\mathcal{O}(\Pop)$ with normal vector $\mathbf{u}_{p<q}$ is linear equivalent to the order polytope $\mathcal{O}((\tilde\Po)^{op})$, where $\tilde\Po$ is the poset obtained by first contracting the edge $p<q$ in the Hasse diagram of $\hat\Po$ and then removing $\hat 0$ and $\hat 1$ (see \cite{St86}). Therefore it follows from \cite[Prop. 3.2.9]{CLS11} that $D_{p<q}$ is isomorphic to the Hibi variety $X_{\tilde\Po}$. More explicitly, we have $D_{p<q}=X_\mathcal{P}\cap V(x_I:|(I\cup\{\hat{0}\})\cap\{p,q\}|=1)\subseteq\mathbb{P}^{|\mathcal{I}(\mathcal{P})|-1}$.
\end{rem}


Let $\textnormal{Cl}(X_\mathcal{P})$ denote the divisor class group of $X_\mathcal{P}$. The main result of this section is the following.

\begin{thm}\label{Cl}
Let $\mathcal{P}$ be a finite poset with $n$ elements and $X_\mathcal{P}$ the associated projective Hibi variety. Then we have
\begin{equation*}
\textnormal{Cl}(X_\mathcal{P})\cong\mathbb{Z}^{|\mathcal{C}(\hat{\mathcal{P}})|-n}.
\end{equation*}
\end{thm}
\begin{proof}
 We have the well-known exact sequence (see e.g. \cite[Thm. 4.1.3]{CLS11})
\begin{equation*}	
0\longrightarrow\mathbb{Z}^\mathcal{P}\xlongrightarrow{\phi}\textnormal{Div}_T(X_\mathcal{P})\longrightarrow\textnormal{Cl}(X_\mathcal{P})\longrightarrow 0
\end{equation*}
where the second map sends a divisor $D$ to its divisor class $[D]$ and $\phi$ is defined by
\begin{equation*}
\phi(\mathbf{m})=\sum\limits_{p<q\in\mathcal{C}(\hat{\mathcal{P}})}{\langle \mathbf{m},\mathbf{u}_{p<q}\rangle D_{p<q}}.
\end{equation*}
More explicitly, we have
\begin{equation}\label{phi}
\phi(\mathbf{e}_p)=\sum\limits_{p<q\in\mathcal{C}(\hat{\mathcal{P}})}{D_{p<q}}-\sum\limits_{r<p\in\mathcal{C}(\hat{\mathcal{P}})}{D_{r<p}}.
\end{equation}
To prove the theorem we will define a map $\psi:\textnormal{Div}_T(X_\mathcal{P})\to\mathbb{Z}^{|\mathcal{C}(\hat{\mathcal{P}})|-n}$ such that the sequence
\begin{equation*}
0\longrightarrow\mathbb{Z}^{|\mathcal{P}|}\xlongrightarrow{\phi}\textnormal{Div}_T(X_\mathcal{P})
\xlongrightarrow{\psi}\mathbb{Z}^{|\mathcal{C}(\hat{\mathcal{P}})|-n}\longrightarrow 0
\end{equation*}
is exact. From this it follows that $\textnormal{Cl}(X_{\mathcal{P}})\cong\mathbb{Z}^{|\mathcal{C}(\hat{\mathcal{P}})|-n}$.\\
To define $\mathcal{\psi}$ we do the following. For every $p\in\mathcal P$ we choose an element $r_p\in\mathcal{P}\cup\{\hat 0\}$ such that $r_p<p$ is a covering relation. Let $T$ be the connected subgraph of the Hasse diagram of $\mathcal{P}\cup\{\hat 0\}$ whose edges are the covering relations $r_p<p$ for all $p\in\mathcal P$. Since $T$ has $n$ edges we can define a basis of $\mathbb{Z}^{|\mathcal{C}(\hat{\mathcal{P}})|-n}$ of the form $\{\mathbf{e}_{p<q}:p<q\in\mathcal{C}(\hat{\mathcal{P}})\backslash T\}$. Now define $\psi(D_{p<q})=\mathbf{e}_{p<q}$ for $p<q\in\mathcal{C}(\hat{\mathcal{P}})\backslash T$. We want to define the image of all other divisors in a way such that $\textnormal{im}(\phi)\subseteq \textnormal{ker}(\psi)$. From \eqref{phi} we get that for $p<q\in\mathcal{C}(\hat{\mathcal{P}})$ we must have
\begin{equation}\label{psi}
\psi(D_{p<q})=\sum\limits_{q<r\in\mathcal{C}(\hat{\mathcal{P}})}{\psi(D_{q<r})}-
\sum\limits_{p'<q\in\mathcal{C}(\hat{\mathcal{P}}):p'\neq p}{\psi(D_{p'<q})}.
\end{equation}
If $q$ is a leaf of $T$ equation \eqref{psi} uniquely defines $\psi(D_{p<q})$. But in fact, as we see by inductively removing leaves, the condition in \eqref{psi} already determines the value of $\psi$ on all edges of $T$.\\
It remains to show that $\textnormal{ker}(\psi)\subseteq\textnormal{im}(\phi)$. Let $D=\sum\limits_{p<q\in\mathcal{C}(\hat{\mathcal{P}})}{\alpha_{p<q}D_{p<q}}$ be a divisor in $\textnormal{ker}(\psi)$. We claim that it suffices to find $\mathbf{m}\in\mathbb{Z}^{\mathcal{P}}$ such that for $D'=D+\phi(\mathbf{m})=\sum\limits_{p<q\in\mathcal{C}(\hat{\mathcal{P}})}{\alpha'_{p<q}D_{p<q}}$ we have $\alpha'_{p<q}=0$ whenever $p<q\in T$. Indeed, by the first part of the proof, $D'$ must lie in $\textnormal{ker}(\psi)$. But since $\alpha'_{p<q}=0$ for all $p<q\in T$ this implies $D'=0$ and therefore $D=-\phi(\mathbf{m})\in\textnormal{im}(\phi)$. 

Any such $\mathbf{m}$ has to satisfy
\begin{equation*}
0=\alpha'_{r_p<p}=
\begin{cases}
\alpha_{r_p<p}-m_p & \textnormal{ if }r_p=\hat 0\textnormal{ and}\\
\alpha_{r_p<p}+m_{r_p}-m_p & \textnormal{ otherwise}.
\end{cases}
\end{equation*}
Hence we define $\mathbf{m}=(m_p)_{p\in\mathcal P}$ inductively by 
\begin{equation*}
m_p=
\begin{cases}
\alpha_{\hat 0<p} & \textnormal{ for }p \textnormal{ minimal element of }\mathcal{P}\textnormal{ and}\\
\alpha_{r_p<p}+m_{r_p} & \textnormal{ otherwise}.
\end{cases}
\end{equation*}
It is easy to see that this $\mathbf{m}$ has the desired properties.
\end{proof}

From the proof of Theorem \ref{Cl} we immediately get the following description of generators of $\textnormal{Cl}(X_\mathcal{P})$.

\begin{cor}\label{gens}
Let $T$ be an arborescence in the Hasse diagram of $\Po\cup\{\hat 0\}$, i.e. a subgraph which for every $p\in\Po$ contains a unique directed path from $\hat 0$ to $p$. Then the divisor class group $\textnormal{Cl}(X_\mathcal{P})$ is the free abelian group generated by the divisor classes $\{[D_{p<q}]:p<q\in\mathcal{C}(\hat{\mathcal{P}})\backslash T\}$. 
\end{cor}

\begin{rem}
The above proof is similar to the one in \cite{HHN92}, where the divisor class group of \emph{affine} Hibi varietes is computed.
\end{rem}

\section{Picard Group}

Let $\textnormal{Pic}(X_\mathcal{P})$ denote the Picard group of $X_{\Po}$. The main result of this section is the following.

\begin{thm}\label{Pic}
We have $\textnormal{Pic}(X_\mathcal{P})\cong\mathbb{Z}^l$ where $l$ denotes the number of connected components of the Hasse diagram of $\mathcal{P}$.
\end{thm}

The Picard group $\textnormal{Pic}(X_\mathcal{P})$ is isomorphic to the subgroup of $\textnormal{Cl}(X_\mathcal{P})$ which consists of divisor classes of locally principal divisors. Hence, we want to understand when a divisor $D_{p<q}$ is locally principal.

For an ideal $I\in\mathcal{I}(\mathcal{P})$ let $\mathcal{C}_I(\hat{\mathcal{P}})=\{p<q\in\mathcal{C}(\hat{\mathcal{P}}):|\{p,q\}\cap(I\cup\{\hat 0\})|\neq 1\}$. Note that $\mathcal{C}_I(\hat{\mathcal{P}})$ corresponds to the set of all facets of $\mathcal{O}(\Pop)$ which contain the vertex $\mathbf{a}_I$. Recall the description of the facet normals given in equation \eqref{FacetNormals}. With this notation we have the following criterion, which is a consequence of Thm. 4.2.8. in \cite{CLS11}.

\begin{lem}\label{LocPrinc}
Let $D=\sum\limits_{p<q\in\mathcal{C}(\hat{\mathcal{P}})}{\alpha_{p<q}D_{p<q}}$. Then $D$ is locally principal if and only if for every $I\in\mathcal{I}(\mathcal{P})$ there is $\mathbf{m}\in\mathbb{Z}^{\mathcal{P}}$ such that 
\begin{equation*}
\langle \mathbf{m},\mathbf{u}_{p<q}\rangle=\alpha_{p<q}\textnormal{ for all }p<q\in\mathcal{C}_I(\hat{\mathcal{P}}).
\end{equation*}
\end{lem}

We will now use this to prove the main theorem.

\begin{proof}[Proof of Theorem \ref{Pic}]
We want to describe the subgroup of $\textnormal{Cl}(X_\mathcal{P})$ which consists of divisor classes of locally principal divisors. Let $[D]$ be a divisor class such that $D$ is locally principal. By Corollary \ref{gens} we may assume that $D$ is of the form
\begin{equation*}
D=\sum\limits_{p<q\in\mathcal{C}(\hat{\mathcal{P}})\backslash T}{\alpha_{p<q}D_{p<q}}.
\end{equation*} 
We will first apply Lemma \ref{LocPrinc} for the ideals $\mathcal{P}$ and $\emptyset$ to get some conditions on the coefficients $\alpha_{p<q}$. Then we will show that these conditions are in fact sufficient.\\
Let $I=\mathcal{P}\in\mathcal{I}(\mathcal{P})$. Then $\mathcal{C}_I(\hat{\mathcal{P}})=\{p<q\in\mathcal{C}(\hat{\mathcal{P}}): q\neq\hat 1\}$. We claim that for all $p<q\in\mathcal{C}_I(\hat{\mathcal{P}})$ we must have $\alpha_{p<q}=0$. First note that for any chain $\hat 0< p_1<\cdots< p_k< q$ in the Hasse diagram of $\hat{\mathcal{P}}$ we have by the above lemma
\begin{equation*}
\alpha_{\hat 0<p_1}+\sum\limits_{1\le i\le k-1}{\alpha_{p_i<p_{i+1}}}+\alpha_{p_k<q}=-m_q=0,
\end{equation*}
where the last equality follows from choosing a chain in $T$. Now consider a chain of the form $\hat 0< p_1'<\cdots< p_l'< p< q$ such that $\hat 0< p_1'<\cdots< p_l'< p$ lies in $T$. This yields $\alpha_{p<q}=0$.\\
So far we have shown that $D$ must be of the form $D=\sum_{p\in M}{\alpha_{p<\hat 1}D_{p<\hat 1}}$, where $M$ denotes the set of maximal elements of $\mathcal{P}$. Now choose $I=\emptyset\in\mathcal{I}(\mathcal{P})$. We have $\mathcal{C}_I(\hat{\mathcal{P}})=\{p<q\in\mathcal{C}(\hat{\mathcal{P}}):p\neq\hat 0\}$. We claim that if $p_1,p_2\in M$ are in the same connected component of $\mathcal{P}$ then we must have $\alpha_{p_1<\hat 1}=\alpha_{p_2<\hat 1}$. We call $p_1,p_2\in M$ adjacent if there exists a $q\in\mathcal{P}$ such that $q<p_1$ and $q<p_2$. Since $\mathcal{P}$ is finite it suffices to prove the claim for adjacent $p_1,p_2$. Let $q\in\mathcal{P}$ such that $q<p_1,p_2$. As above we get $0=m_q-m_{p_1}=m_q-m_{p_2}$, which in particular implies $m_{p_1}=m_{p_2}$. But $m_{p_i}=\alpha_{p_i<\hat 1}$ by Lemma \ref{LocPrinc}, which proves the claim.\\
Let $\mathfrak{C}(\mathcal{P})$ be the set of connected components of $\mathcal{P}$. We have shown that $D$ must be of the form
\begin{equation*}
D=\sum\limits_{C\in\mathfrak{C}(\mathcal{P})}{\alpha_C D_C}\textnormal{ where }D_C=\sum\limits_{p\in M\cap C}{D_{p<\hat 1}}.
\end{equation*}
The only thing left to show is that every such $D$ is locally principal by again using Lemma \ref{LocPrinc}. Let $I\in\mathcal{I}(\mathcal{P})$. Define $\mathbf{m}=(m_p)_{p\in\mathcal{P}}$ as follows. For all $p\in I$ set $m_{p}=0$. For all $p\in\mathcal{P}\backslash I$, let $C$ be the connected component that $p$ lies in and set $m_{p}=\alpha_C$. It is easy to check that $\mathbf{m}$ has all the desired properties.
 
\end{proof}

\textbf{Acknowledgements.} This work generalizes results from the author's master thesis supervised by Gunnar Fl\o ystad, whom the author would like to thank for fruitful discussions and constant support. The author would also like to thank Raman Sanyal for helpful comments on a previous version of this paper.

\linespread{1.0}
\setlength{\parskip}{0cm}

\small

\begin{thebibliography}{}

\bibitem[BC03]{BC03}
W. Bruns, A. Conca, {\em Gr\"obner Bases and Determinantal Ideals}, Commutative Algebra, Singularities and Computer Algebra, NATO Science Series \textbf{115}, p. 9-66, 2003.

\bibitem[Co95]{Co95}
A. Conca, {\em Ladder Determinantal rings}, Journal of Pure and Applied Algebra \textbf{98}, p. 119-134, 1995.

\bibitem[CLS11]{CLS11}
D. Cox, J. Little, H. Schenck, {\em Toric Varieties}, Graduate Studies in Mathematics \textbf{124}, American Mathematical Society, Providence, RI, 2011.

\bibitem[EH12]{EH12}
V. Ene, J. Herzog, {\em Gr\"obner Bases in Commutative Algebra}, Graduate Studies in Mathematics \textbf{130}, American Mathematical Society, Providence, RI, 2012.

\bibitem[Fr13]{Fr13}
T. Friedl, {\em Projective Hibi Varieties}, Master Thesis, LMU M\"unchen, 2013.

\bibitem[Fu93]{Fu93}
W. Fulton, {\em Introduction to Toric Varieties}, Annals of Mathematical Studies \textbf{131}, Princeton University Press, Princeton, 1993.

\bibitem[GL97]{GL97}
N. Gonciulea, V. Lakshmibai, {\em Schubert Varieties, Toric Varieties and Ladder Determinantal Varieties}, Annales de l'Institut Fourier \textbf{47}, no.4, p. 1013-1064, Grenoble, 1997.

\bibitem[HHN92]{HHN92}
M. Hashimoto, T. Hibi, A. Noma {\em Divisor Class Groups of Affine Semigroup Rings Associated with Distributive Lattices}, Journal of Algebra \textbf{149}, 1992.

\bibitem[Hi87]{Hi87}
T. Hibi, {\em Distributive Lattices, Affine Semigroup Rings and Algebras with Straightening Laws}, Commutative Algebra and Combinatorics (M. Nagata and H. Matsumura, Eds.) Adv. Stud. Pure Math. \textbf{11}, p. 93-109, North-Holland, Amsterdam, 1987.

\bibitem[MS05]{MS05}
E. Miller, B. Sturmfels, {\em Combinatorial Commutative Algebra}, Graduate Texts in Mathematics \textbf{227}, Springer, New York, 2005.

\bibitem[St86]{St86}
R.P. Stanley, {\em Two Poset Polytopes}, Discrete and Computational Geoemtry \textbf{1}, p. 9-23, New York, 1986.

\bibitem[St96]{St96}
B. Sturmfels, {\em Gr\"obner Bases and Convex Polytopes}, AMS University Lecture Series \textbf{8}, American Mathematical Society, Providence, RI, 1996.

\bibitem[Wa96]{Wa96}
D. Wagner, {\em Singularities of Toric Varieties Associated with Finite Distributive Lattices}, Journal of Algebraic Combinatorics \textbf{5}, p. 149-165, 1996.

 
\end{thebibliography}


\end{document}